\documentclass{amsart}

\usepackage{latexsym,amsfonts,exscale,enumerate}
\usepackage{amsmath,amsthm,amscd,amssymb}
\usepackage{hyperref}

\usepackage{mathrsfs}
\usepackage{youngtab}

\usepackage{pstricks}

\psset{linewidth=0.3pt,dimen=middle}
\psset{xunit=.70cm,yunit=0.70cm}

\usepackage[all]{xy}
\SelectTips{cm}{}

\usepackage{graphicx}

\normalfont\upshape

\newcommand{\Hq}{\mathcal{H}_q(n)}
\newcommand{\Ho}{\mathcal{H}_{-1}(n)}

\newcommand{\Hz}{\mathcal{H}_{0}(n)}

\newcommand{\NH}{NH}
\newcommand{\ONH}{ONH}

\newcommand{\pol}{\mathrm{Pol}}
\newcommand{\opol}{\mathrm{OPol}}

\newcommand{\sym}{\Lambda}
\newcommand{\osym}{\mathrm{O}\Lambda}

\newcommand{\qrk}{\mathrm{rk}_q}

\newcommand{\sch}{\mathfrak{s}}

\newcommand{\undx}{\underline{x}}

\newcommand{\xt}{\widetilde{x}}

\def\nn{\notag}

\newcommand{\refequal}[1]{\xy {\ar@{=}^{#1}
(-1,0)*{};(1,0)*{}};
\endxy}

\renewcommand{\to}{\rightarrow}
\newcommand{\maps}{\colon}

\newcommand{\im}{{\rm im\ }}

\newcommand{\rk}{{\rm rk\ }}

\newcommand{\scs}{\scriptstyle}

\def\mf{\mathfrak}
\def\shuffle{\,\raise 1pt\hbox{$\scriptscriptstyle\cup{\mskip
               -4mu}\cup$}\,}

\def\cal#1{\mathcal{#1}}%
\def\1{\mathbbm{1}}%

\theoremstyle{definition}
\newtheorem{thm}{Theorem}[section]
\newtheorem*{thmm}{Theorem}
\newtheorem{cor}[thm]{Corollary}

\newtheorem{lem}[thm]{Lemma}
\newtheorem{rem}[thm]{Remark}
\newtheorem{prop}[thm]{Proposition}
\newtheorem{defn}[thm]{Definition}
\newtheorem{example}[thm]{Example}

\numberwithin{equation}{section}

\def\emph#1{{\sl #1\/}}

\let\tilde=\widetilde

\let\phi=\varphi
\let\theta=\vartheta

\usepackage{bbm}
\def\C{{\mathbbm C}}

\def\Z{{\mathbbm Z}}
\def\Q{{\mathbbm Q}}

\newcommand{\topp}{{\rm top}}

\title{Oddification of the cohomology of type $A$ Springer varieties}

      \author{Aaron D.\ Lauda}
      \address{Department of Mathematics, University of Southern California, Los Angeles, CA 90089, USA}
\email{lauda@usc.edu}

\author{Heather M.\ Russell}
     \address{Department of Mathematics, University of Southern California, Los Angeles, CA 90089, USA}
\email{heathemr@usc.edu}
\date{March 5, 2012}

\begin{document}
%

\maketitle

\begin{abstract}
We identify the ring of odd symmetric functions introduced by Ellis and Khovanov as the space of skew polynomials fixed by a natural action of the Hecke algebra at $q=-1$.  This allows us to define graded modules over the Hecke algebra at $q=-1$  that are `odd' analogs of the cohomology of type $A$ Springer varieties.  The graded module associated to the full flag variety corresponds to the quotient of the skew polynomial ring by the left ideal of nonconstant odd symmetric functions.  The top degree component of the odd cohomology of Springer varieties is identified with the corresponding Specht module of the Hecke algebra at $q=-1$.
\end{abstract}


%
\section{Introduction}
%

The full flag variety $X$ consists of the set of all flags in $\C^n$,
\[
 X = \left\{ 0=U_0 \subset U_1 \subset \dots \subset U_n=\C^n  \mid \dim_{\C} U_i = i \right\}.
\]
The integral cohomology ring $H(X):= H^*(X,\Z)$ can be explicitly described as a quotient of the graded polynomial ring $\pol_n:= \Z[x_1, \dots, x_n]$ with $\deg(x_j)=2$.  Write  $\varepsilon_r=\varepsilon_r(x_1,\dots, x_n)$ for the $r$th elementary symmetric function in $n$ variables. Denote by $\sym_n:= \Z[x_1,\dots, x_n]^{S_n}=\Z[\varepsilon_1,\varepsilon_2, \dots, \varepsilon_n]$ the ring of symmetric functions, and let $\sym_n^+$ be the two-sided ideal of $\pol_n$ generated by the elementary symmetric functions with no constant term. The cohomology ring of $X$ is the quotient
\begin{equation} \label{eq_intro-full}
H(X)\cong \Z[x_1,\dots,x_n]/ \sym_n^+.
\end{equation}

Let $\lambda= \lambda_1 \geq \dots \geq \lambda_n\geq 0$ be a partition of $n$, and let $\lambda' = \lambda'_1 \geq \dots \geq \lambda'_n\geq 0$ be its transpose partition. Consider a nilpotent matrix $x^{\lambda}$ of Jordan type $\lambda$.
The Springer variety $X^{\lambda}$ is the closed subvariety of $X$ consisting of those flags preserved by $x^{\lambda}$
\[
 X^{\lambda} = \left\{ (U_0,U_1, \dots, U_n)\in X \mid x^{\lambda} U_i \subseteq U_{i-1} \right\}.
\]
Work of De Concini and Procesi~\cite{DP} and Tanisaki~\cite{Tan} generalizes the above characterization of $H(X)$ showing that the integral cohomology ring of the Springer variety $H(X^{\lambda}):=H^*(X^{\lambda}, \Z)$ is isomorphic to the quotient of the polynomial ring $\pol_n$ by the two-sided ideal $I_{\lambda}$ generated by the following elements:
\begin{equation} \label{def_Ilambda}
 \cal{C}_{\lambda} := \left\{ \varepsilon_r^S
\;\middle\vert\;
\begin{array}{l}
    k \geq r > k -\delta_k(\lambda),  \\
    S\subseteq \{x_1, \ldots, x_n\}, |S|=k
\end{array}
 \right\},
\end{equation}
where $\varepsilon_r^S$ denotes the $r$th elementary symmetric function in the variables of $S$ and $$\delta_k(\lambda)=\lambda_n'+\dots+\lambda_{n-k+1}'.$$  For more details on the cohomology of Springer varieties see ~\cite{BO,GP,HSp,Springer}, and for a generalization of this ideal see \cite{MT}.

The cohomology rings of Springer varieties carry additional structure summarized below.

\begin{enumerate}
  \item There is an action of the symmetric group $S_n$ on $H(X^{\lambda})$ that preserves the cohomological grading.

  \item The pull-back homomorphism $i^{\ast} \maps H(X)\to H(X^{\lambda})$ arising from the inclusion $X^{\lambda} \to X$ is a surjective $\Z S_n$-module map.

  \item   The top nonzero degree of rational cohomology $H^{\topp}(X^{\lambda},\Q)$ is isomorphic to the irreducible $\Q S_n$-module corresponding to the partition $\lambda$.

  \item The full rational cohomology $H^*(X^{\lambda}, \Q)$ is isomorphic as a $\Q S_n$-module to the permutation module $M^{\lambda}$.
\end{enumerate}

From the polynomial presentation of the cohomology ring of the full flag variety in \eqref{eq_intro-full} the action of the symmetric group is quite transparent (Springer's original action is more sophisticated~\cite{Springer}).  In particular, $w \in S_n$ acts on $f \in \pol_n$ by permuting indices so that
\[
w \cdot f(x_1,x_2, \dots , x_n) = f(x_{w(1)}, x_{w(2)}, \dots , x_{w(n)})
.\]
Since this action of $S_n$ preserves the ideal of symmetric functions $\sym_n^+$ , it descends to an action on the quotient $H(X)$.  The ideal $I_{\lambda}$  is preserved by the action of $S_n$, so there is also a natural $S_n$-action on the quotient $H(X^{\lambda}) = \pol_n / I_{\lambda}$.  The rings $H(X^{\lambda})$ inherit the grading from $\pol_n$. Since the ideal $I_{\lambda}$ is homogeneous, the action of $S_n$ restricts to an action on each graded piece of $H(X^{\lambda})$.

\bigskip
Recent work of Ellis and Khovanov introduces a new ``odd" analog of the Hopf algebra of symmetric functions~\cite{EK}.  The ring of odd symmetric functions, denoted by $\osym_n$, is noncommutative despite exhibiting behaviour surprisingly similar to the ordinary ring of symmetric functions; it has odd analogs of elementary, complete, and Schur function bases. The ring $\osym_n$ is a certain quotient of a related version of noncommutative symmetric functions defined by Thibon and Ung~\cite{TU}.  The precise form of this quotient is essential to the results in this article.

The odd symmetric function and ordinary symmetric function rings have the same graded rank and become isomorphic when coefficients are reduced modulo two.  These quintessential rank and reduction features characterize the ``oddification" of an algebraic object. Recent work of the first author with Ellis and Khovanov gives an explicit realization of odd symmetric functions in the ring of skew polynomials
\begin{equation}
\opol_n=\Z\langle x_1,\ldots,x_n\rangle/\langle x_ix_j+x_jx_i=0\text{ for }i\neq j\rangle,
\end{equation}
and uses this to connect odd symmetric functions to a new odd analog of the nilHecke algebra~\cite{EKL} inspired by constructions in Heegaard-Floer homology~\cite{LOT}.    Kang, Kashiwara, and Tsuchioka~\cite{KKT} independently studied the odd nilHecke algebra in their work on categorification of Hecke-Clifford super algebras, and it is closely related to earlier work of Wang~\cite{Wang}.  This algebra also appears in recent work of Hill and Wang on categorifications of quantum Kac-Moody superalgebras~\cite{HW}.

The parallels between the odd and even nilHecke algebras extend further.
Just as cyclotomic quotients of the nilHecke algebra are Morita equivalent to the integral cohomology rings $H^{\ast}(Gr(k,N),\Z)$ of Grassmannians of $k$-planes in $\C^N$ ~\cite[Section 5]{Lau-exp}, the corresponding cyclotomic quotients of the odd nilHecke algebra are Morita equivalent to {\em noncommutative} rings $OH(Gr(k,N))$~\cite{EKL}.  The graded ranks of these rings are the same, they become isomorphic when coefficients are reduced modulo two, and $OH(Gr(k,N))$ has a basis consisting of certain odd Schur polynomials ~\cite{EKL}.  In other words, $OH(Gr(k,N))$ is an oddification of the cohomology of the Grassmannian.

The nilHecke algebra, cohomology rings of Grassmannians, and the ring of symmetric functions are vital components in the theory of categorified quantum $\mathfrak{sl}_2$~\cite{BFK,FKS,CR,Lau1,Lau2,KLMS}, which underlies Khovanov homology.  All of the odd structures discussed above aim to establish a similar algebraic framework underlying the odd Khovanov homology theory of Ozsv\'{a}th, Rasmussen, Szab\'{o}~\cite{ORS}. These invariants, which both categorify the Jones polynomial, are distinct (see ~\cite[Section 3.1]{Shum}) but agree modulo two.

Another algebraic structure closely connected to Khovanov homology is the Springer variety $X^{(n,n)}$. In particular, the cohomology of $X^{(n,n)}$ is isomorphic to the center of Khovanov's arc algebra $H_n$ used in one construction of Khovanov homology for tangles ~\cite{KhSp}.  There is a generalization of Khovanov's arc algebra~\cite{Strop,ChK} that Stroppel and Webster showed is isomorphic to the convolution algebra formed on cohomologies of components of the Springer variety~\cite{SW}. These Springer varietes also arise in  Cautis-Kamnitzer's geometric construction of Khovanov homology using derived categories coherent sheaves~\cite{CK0}.  These connections lead one to ask whether the cohomology of Springer varieties also have natural ``odd" analogs.  In this article we introduce graded quotient modules over the skew polynomial ring which are oddifications of the cohomology of Springer varieties $X^{\lambda}$ for any partition.

In the odd setting there is  an action of the symmetric group $S_n$ on $\opol_n$.  However,   odd symmetric functions are not preserved by this action. This raises the question of what object plays the role of the symmetric group $S_n$ in the odd theory of Springer varieties.  Here we show that the  Hecke algebra $\Hq$ at $q=-1$ is the natural analog.

We begin by defining an action of the Hecke algebra $\Hq$ at $q=-1$ on the space of skew polynomials $\opol_n$. A related action of $\Hq$ on $\pol_n$ was defined in \cite{APR} where $T_i$ acts by the q-commutator
\begin{equation}
  [\partial_i,x_i]_q=\partial_i x_i - q x_i \partial_i,
\end{equation}
and $\partial_i$ is the divided difference operator (see section~\ref{subsec_hecke}).  When $q=1$ the q-commutator action is just the usual permutation action of $\mathcal{H}_{1}(n)=\Z[S_n]$ on $\pol_n$.  In \cite[Claim 3.1]{APR} it is shown that the ring of symmetric functions $\sym_n$ can also be understood as the space of invariants in $\pol_n$ under this action of $\Hq$.  The odd nilHecke algebra acts  on skew polynomials by so-called odd divided difference operators.  Using these operators we define an action of $\Hq$ at $q=-1$ by a similar $q$-commutator formula and prove the following theorem.

\begin{thmm}
Odd symmetric functions $\osym_n \subset \opol_n$ are precisely the skew-polynomials in $\opol_n$ that are invariant under the action of $\Ho$. In other words,  $T_i(f) = f$ for all generators $T_i$ in $\Ho$ if and only if $f$ is odd symmetric.
\end{thmm}

Just as commutativity is lost  in the passage from cohomology rings of Grassmanians to odd cohomology of Grassmanians, the odd cohomology of Springer varieties no longer possess a commutative ring structure or even a ring structure at all. This occurs because the left ideal $OI^L:=\opol_n \cdot \osym^+$ generated by the elements in $\osym^+$ differs from the right ideal $OI^R:= \osym^+ \cdot \opol_n$. This fact was first observed by Alexander Ellis.  Thus, the definition  of the odd cohomology of the full flag variety requires a choice of the left or right ideal.  Quotienting by the two-sided ideal leads to torsion in the quotient ring, and the resulting ring will not have the same dimension as the ordinary cohomology ring $H(X)$ when coefficients are reduced modulo two.

The action of $\Ho$ on $\opol_n$ is homogeneous with the respect to the grading.  Since
the left ideal $OI^{L}$ is preserved by the action of $\Ho$ and is generated by homogenous elements, the quotient module $\opol_n / OI^L$ is a left graded $\Ho$-module.   We denote this quotient by $OH(X)$ since it will play the role of the cohomology of the full flag variety in the odd theory.

Taking the analogy between the even and odd theories one step further, we construct odd analogs $OI_{\lambda}$ of the ideals $I_{\lambda}$ from \eqref{def_Ilambda} and show that these are also invariant under the Hecke algebra $\Ho$ action.    This allows us to define graded left $\Ho$-modules
\begin{equation} \nn
  OH(X^{\lambda}) := \opol_n/ OI_{\lambda}.
\end{equation}
The ring $\opol_n$ is graded with $\deg(x_i)=2$, and the ideal $OI_{\lambda}$ is homogeneous with respect to this grading,  so that the action of $\Ho$ restricts to an action on each homogeneous component of $OH(X^{\lambda})$. In this article we prove the following results.

\begin{thmm} \hfill
\begin{enumerate}
\item  There is an isomorphism of graded vector spaces:
\begin{equation} \nn
  OH(X)\otimes_{\Z}  \Z_2 \cong_{} H(X) \otimes_{\Z}\Z_2, \qquad \text{and} \qquad
   OH(X^{\lambda})\otimes_{\Z}\Z_2  \cong H(X^{\lambda})\otimes_{\Z}\Z_2 .
\end{equation}

\item The underlying abelian group of $OH(X^{\lambda})$ is a free graded $\Z$-module that admits an integral basis identical in form to the Garsia-Procesi basis~\cite{GP} of $H(X^{\lambda})$.  Using this basis we show that the graded ranks agree
\begin{equation} \nn
\qrk OH(X^{\lambda}) = \qrk H(X^{\lambda}).
\end{equation}

\item There is a surjective homomorphism $ OH(X) \to OH(X^{\lambda})$ of graded $\Ho$-modules.

\item Working over a field $\mathbbm{F}$, the top graded component $OH^{\topp}(X^{\lambda})$ of $OH(X^{\lambda})$ is isomorphic as an $\Ho$-module to the Specht module $S^{\lambda}_{-1}$.
\end{enumerate}
\end{thmm}

Though we utilize many of the combinatorial techniques introduced by Tanisaki~\cite{Tan} and by Garsia and Procesi~\cite{GP}, we emphasize that the identification of $OH^{\topp}(X^{\lambda})$ with the Specht module requires  new methods.  This phenomenon is genuinely ``odd'' in that arguments make extensive use of calculations in the odd nilHecke algebra.  The original identification of the Specht module with the top degree cohomology of the Springer variety $X^{\lambda}$ uses results from geometry.   Since the precise connection between the odd theory and geometry has yet to be developed, we must prove directly that the top cohomology of odd Springer varieties is isomorphic to the Specht module of the Hecke algebra $\Ho$.  This requires an explicit verification of row relations and Garnir relations (Propositions~\ref{prop_row} and \ref{prop_Garnir}).


\bigskip
\noindent {\bf Acknowledgments:}
This project arose out of a general program started by the first author with Alexander Ellis and Mikhail Khovanov to identify odd analogs of representation theoretic objects.  The idea of looking for odd analogs of cohomology of Springer fibers arose from these conversations.  Both authors are grateful to Monica Vazirani and Alexander Kleshchev for helpful discussions on the representation theory of Hecke algebras, and to Sabin Cautis and Anthony Licata for helpful discussions on Springer varieties.
The first author was partially supported by the NSF grant DMS-0855713, and the Alfred P. Sloan foundation.  We would also like to acknowledge the use of the NC-Algebra Mathematica package for calculations related to this project.

%
\section{Odd symmetric functions as invariant functions}
%

%
\subsection{Reminders from the even theory}
%

%
\subsubsection{The nilHecke ring  (Even case)}
%

The nilHecke ring $\NH_n$ is the graded unital associative ring generated by elements $x_1,\ldots,x_n$ of degree 2 and elements $\partial_1,\ldots,\partial_{n-1}$ of degree $-2$, subject to the relations
\begin{equation}\begin{split} \label{eq_rel_nilHecke}
&\partial_i^2=0,\qquad\partial_i\partial_{i+1}\partial_i=\partial_{i+1}\partial_i\partial_{i+1}\\
&x_i\partial_i+\partial_ix_{i+1}=1,\qquad\partial_ix_i+x_{i+1}\partial_i=1\\
&x_ix_j+x_jx_i=0\quad(i\neq j),\qquad\partial_i\partial_j+\partial_j\partial_i=0\quad(|i-j|>1),\\
&x_i\partial_j+\partial_jx_i=0\quad(|i-j|>1).
\end{split}\end{equation}

The nilHecke algebra $\NH_n$ acts as endomorphisms of the polynomial ring $\pol_n:= \Z[x_1, x_2, \dots , x_n]$ with $x_i$ acting by multiplication and $\partial_i$ acting by divided difference operators. That is for $f \in \pol_n$
\begin{equation}\label{divdiff}
  \partial_i(f) := \frac{f - s_i(f)}{x_i - x_{i+1}},
\end{equation}
where $s_i(f)$ denotes the standard action of the symmetric group $S_n$ on the polynomial ring $\pol_n$ by exchanging the variables $x_i$ and $x_{i+1}$.  Observe that the action of $\partial_i$ on any polynomial $f$ can be deduced from the rules
\begin{equation}
\partial_i(1)=0, \qquad \qquad
\partial_i(x_j)=
\begin{cases}
1&\text{if }j=i\\
-1&\text{if }j \neq i+1\\
0&\text{otherwise,}
\end{cases}
\end{equation}
and the Leibniz rule
\begin{equation}
\partial_i(fg)=\partial_i(f)g+s_i(f)\partial_i(g)\text{ for all }f,g\in \pol_n.
\end{equation}

For any $w \in S_n$ choose a reduced word decomposition $w=s_{i_1} \dots s_{i_m}$  into a product of elementary transpositions.  Define $\partial_{w}=\partial_{i-1} \dots \partial_{i_m}$.  It is clear from the relations \eqref{eq_rel_nilHecke} that $\partial_w$ is independent of the choice of reduced word decomposition for $w$.

%
\subsubsection{Symmetric functions (Even case)}
%

Throughout the paper we refer to symmetric functions and symmetric polynomials interchangeably. The ring of symmetric polynomials can be understood in several ways:
\begin{enumerate}
  \item $\sym_n$ is the space of invariants for the standard action of $S_n$ on $\pol_n$, i.e.
  $\sym_n = \Z[x_1, \dots , x_n] ^{S_n}$.
  \item $\sym_n$ is the intersection of the kernels or images of divided difference operators:
  \begin{equation}
 \nn \sym_n=\bigcap_{i=1}^{n-1}\ker(\partial_i)=\bigcap_{i=1}^{n-1}\im(\partial_i).
\end{equation}
\end{enumerate}

The ring of symmetric functions $\sym_n$ can be given a concrete presentation using the elementary symmetric functions in $\pol_n$. Below we label these functions  `even' to contrast with the odd theory being studied in this paper.   Recall that the ring of symmetric polynomials is isomorphic to a graded polynomial ring,
\begin{equation}
\sym_n\cong\Z[\varepsilon_1^\text{even},\varepsilon_2^\text{even},\ldots,\varepsilon_n^\text{even}],
\end{equation}
where $\deg(\varepsilon_k^\text{even})=2k$ and
\begin{equation}
\varepsilon_k^\text{even}= \varepsilon_k^\text{even}(x_1,\ldots,x_n)=\sum_{1\leq i_1<\cdots<i_k\leq n}x_{i_1}\cdots x_{i_n}.
\end{equation}

The polynomial ring $\pol_n$ is a free $\sym_n$ module of rank $n!$ with basis over $\sym_n$ consisting of \textit{Schubert polynomials} $\sch_w^{even}\in\pol_n$.  These polynomials are defined for  $w\in S_n$ as
\begin{equation}
\sch_w(x_1,\ldots,x_n)^{even}=\partial_{w^{-1}w_0}(\undx^{\delta_n}),
\end{equation}
where $w_0$ is the longest element of $S_n$ and $\undx^{\delta_n}:=x_1^{n-1}x_2^{n-2} \dots x_n^0$.   The degree of $\sch_w$ is $2\ell(w)$.

%
\subsubsection{Hecke algebra action on $\pol_n$} \label{subsec_hecke}
%

The ring of symmetric functions also arises from a more general action of the Hecke algebra $\Hq$ of type $A$.  The Hecke algebra $\Hq$ is defined over and commutative domain $R$ with $1$ where $q$ is any element of $R$.  Here we work integrally whenever possible taking $R=\Z$.  In section~\ref{sec-Specht} we take $R=\mathbb{F}$ for a field $\mathbb{F}$.

The 0-Hecke algebra $\Hz$ is the  unital associative $R$-algebra with generators $\overline{\partial}_1, \dots , \overline{\partial}_{n-1}$ satisfying the relations
\begin{equation}\begin{split}
&\overline{\partial}_i^2=\overline{\partial}_{i}, \\
&\overline{\partial}_i\overline{\partial}_j=\overline{\partial}_j\overline{\partial}_i\quad(i\neq j),\\
&\overline{\partial}_{i+1}\overline{\partial}_i\overline{\partial}_{i+1}
=\overline{\partial}_{i+1}\overline{\partial}_i\overline{\partial}_{i+1}
.
\end{split}\end{equation}
The 0-Hecke algebra acts on $\pol_n$ by the so-called isobaric divided difference operators $\overline{\partial}_{i}(f)= \partial_i x_i(f)$. The 0-Hecke algebra is the $q=0$ specialization of the Hecke algebra associated to $S_n$. The Hecke algebra $\Hq$ is the $R$-algebra with generators $T_i$ for $1 \leq i \leq n-1$ and relations
\begin{equation} \label{eq_hecke-def}\begin{split}
&T_i^2=(1-q)T_{i}+q, \\
&T_iT_j=T_jT_i\quad(|i- j|>1),\\
&T_{i+1}T_iT_{i+1}
=T_{i+1}T_iT_{i+1}
.
\end{split}\end{equation}
Here we follow \cite{APR} and use a nonstandard presentation of the Hecke algebra so that our specialization at $q=0$ agrees with the form of the 0-Hecke algebra defined above. The usual conventions replace $T_i$ with $-T_i$.

In \cite[Claim 3.1]{APR} it is shown that the Hecke algebra of type $A$ at generic $q$ acts on the polynomial ring $\pol_n$ with $T_i$ acting by the q-commutator $[\partial_i,x_i]_q=\partial_i x_i - q x_i \partial_i$. When $q=1$ the q-commutator is just the standard action of the symmetric group by permuting variables.  When $q=0$ the $q$-commutator is the isobaric divided difference operator $\overline{\partial}_i$ of the 0-Hecke algebra. It is also shown in \cite[Claim 3.2]{APR} that the ideal generated by the ring of symmetric functions $\sym_n$ is invariant under this action.  This observation will be the starting point for our study of odd analogs of the cohomology rings of Springer varieties.

%
\subsection{The odd Theory}
%

In this section we review the theory of the odd nilHecke ring and odd symmetric functions from ~\cite{EK,EKL}, see also \cite{KKT}

%
\subsubsection{Odd NilHecke}
%

Define the \text{odd nilHecke ring} $\ONH_n$ to be the graded unital associative ring generated by elements $x_1,\ldots,x_n$ of degree 2 and elements $\partial_1,\ldots,\partial_{n-1}$ of degree $-2$, subject to the relations
\begin{equation}\begin{split}
&\partial_i^2=0,\qquad\partial_i\partial_{i+1}\partial_i=\partial_{i+1}\partial_i\partial_{i+1},\\
&x_i\partial_i+\partial_ix_{i+1}=1,\qquad\partial_ix_i+x_{i+1}\partial_i=1,\\
&x_ix_j+x_jx_i=0\quad(i\neq j),\qquad\partial_i\partial_j+\partial_j\partial_i=0\quad(|i-j|>1),\\
&x_i\partial_j+\partial_jx_i=0\quad(|i-j|>1).
\end{split}\end{equation}

Define the ring of \textit{skew polynomials} to be the free unital associative algebra on skew-commuting variables $x_1,\ldots,x_n$,
\begin{equation}
\opol_n=\Z\langle x_1,\ldots,x_n\rangle/\langle x_ix_j+x_jx_i=0\text{ for }i\neq j\rangle.
\end{equation}
The symmetric group $S_n$ acts on the degree $k$ part of $\opol_n$ as the tensor product of the permutation representation and the $k$-th tensor power of the sign representation.  That is, for $1\leq i\leq n$, the transposition $s_i\in S_n$ acts as the ring endomorphism
\begin{equation}\label{eqn-S-action}
s_i(x_j)=\begin{cases}
-x_{i+1}&\text{if }j=i\\
-x_i&\text{if }j=i+1\\
-x_j&\text{otherwise.}
\end{cases}
\end{equation}
The \text{odd divided difference operators} are the linear operators $\partial_i$ ($1\leq i\leq n-1$) on $\Z\langle x_1,\ldots,x_n\rangle$ defined by
\begin{equation}
\partial_i(1)=0, \qquad \qquad
\partial_i(x_j)=\begin{cases}
1&\text{if }j=i,i+1\\
0&\text{otherwise,}
\end{cases}
\end{equation}
and the Leibniz rule
\begin{equation} \label{eq_Leibniz}
\partial_i(fg)=\partial_i(f)g+s_i(f)\partial_i(g)\quad \text{ for all }f,g\in\Z\langle x_1,\ldots,x_n\rangle.
\end{equation}
Note that there is no analog of Equation \ref{divdiff} in the odd setting. It is easy to check from the definition of $\partial_i$ that for all $i$
$\partial_i(x_jx_k+x_kx_j)=0\text{ for }j\neq k$,
so $\partial_i$ descends to an operator on $\opol_n$. Considering $\partial_i$ and (multiplication by) $x_j$ as operators on $\opol_n$ defines an action of $\ONH_n$ on $\opol_n$, see \cite{EKL}.

For $w = s_{i_1} \dots s_{i_m}$ in $S_n$  define $\partial_w= \partial_{i_1} \dots \partial_{i_m}$.  Unlike the even case, in the odd nilHecke algebra the element $\partial_{w}$ depends on the choice of reduced expression for $w$ up to a sign.    For the longest element $w_0$ of $S_n$ we fix a particular choice of reduced expression,
\begin{equation*}
\partial_{w_0}=\partial_1(\partial_2\partial_1)\cdots(\partial_{n-1}\cdots\partial_1).
\end{equation*}
For each $w\in S_n$ define the corresponding \textit{odd Schubert polynomial} $\sch_w\in\opol_n$ by
\begin{equation}
\sch_w(x_1,\ldots,x_n)=\partial_{w^{-1}w_0}(\undx^{\delta_n}).
\end{equation}
Again, the degree of $\sch_w$ is $2\ell(w)$.  For $w,w'\in S_n$, the formula
\begin{equation}\label{eqn_compose_oddops}
\partial_w\partial_{w'}=\begin{cases}\pm\partial_{ww'}&\text{if }\ell(ww')=\ell(w)+\ell(w')\\
0&\text{otherwise,}\end{cases}
\end{equation}
implies the following action of odd divided difference operators on odd Schubert polynomials
\begin{equation}\label{eqn_oddop_schubert}
\partial_u\sch_w=\begin{cases}\pm \sch_{wu^{-1}}&\text{if }\ell(wu^{-1})=\ell(w)-\ell(u)\\
0&\text{otherwise.}
\end{cases}\end{equation}
Note that the exact signs in these relations can be determined for specific choices of reduced words for the symmetric group elements involved, but we will not need these coefficients in what follows.

%
\subsubsection{Odd symmetric polynomials}
%

Define the ring of \textit{odd symmetric polynomials} to be the subring
\begin{equation}
\osym_n=\bigcap_{i=1}^{n-1}\ker(\partial_i)=\bigcap_{i=1}^{n-1}\im(\partial_i)
\end{equation}
of $\opol_n$.  The $k$th \text{odd elementary symmetric polynomial} is defined as
\begin{equation}\label{eqn-defn-e}
\varepsilon_k(x_1,\ldots,x_n)=\sum_{1\leq i_1<\cdots<i_k\leq n}\widetilde{x}_{i_1}\cdots\widetilde{x}_{i_k},\qquad\text{where }\widetilde{x}_i=(-1)^{i-1}x_i,
\end{equation}
It was shown in ~\cite{EKL} that the following relations hold in the ring $\osym_n$:
\begin{equation}\label{eqn-e-relations}\begin{split}
&\varepsilon_i\varepsilon_{2m-i}=\varepsilon_{2m-i}\varepsilon_i\qquad(1\leq i,2m-i\leq n)\\
&\varepsilon_i\varepsilon_{2m+1-i}+(-1)^i\varepsilon_{2m+1-i}\varepsilon_i
=(-1)^i\varepsilon_{i+1}\varepsilon_{2m-i}+\varepsilon_{2m-i}\varepsilon_{i+1}\qquad(1\leq i,2m-i\leq n-1)\\
&\varepsilon_1\varepsilon_{2m}+\varepsilon_{2m}\varepsilon_1=2\varepsilon_{2m+1}\qquad(1<2m\leq n-1).
\end{split}\end{equation}
Note that the third is the $i=0$ case of the second.  In particular, the ring $\osym_n$ of odd symmetric functions is noncommutative. These relations lead to an isomorphism of graded rings $\osym_n \cong \Z[\varepsilon_1, \dots, \varepsilon_n]$. See \cite{TU} and the references therein for a related version of noncommutative symmetric functions.

Odd symmetric functions also have natural bases corresponding to complete, monomial, and forgotten symmetric functions, as well as a Schur polynomials basis~\cite{EK,EKL}.  There is also an odd analog of the Littlewood-Richardson rule for Schur polynomials developed by Ellis~\cite{Ellis}.  Recall that the odd Schubert polynomials defined above are a homogeneous basis for $\opol_n$ as a free left and right $\osym_n$-module of rank $n!$~\cite[Proposition 2.13]{EKL}.

\begin{rem}
It is not difficult to construct examples showing that odd symmetric functions are not the $S_n$ invariant functions for the action of $S_n$ defined above.  For example, $s_1(\xt_{1}+\xt_{2}+\xt_{3})=(\xt_1+\xt_2-\xt_3)$.
\end{rem}

%
\subsection{Modular reduction}
%

For a graded abelian group $V$ of finite rank in each degree define the graded rank of $V$ as
 \begin{equation}
\qrk(V)=\sum_{i\in\Z}\rk(V_i)q^i.
\end{equation}
Likewise, for a graded vector space $W$ over a field $\mathbbm{F}$ of finite dimension in each degree, define the graded dimension of $W$ as
\begin{equation}
\mathrm{dim}_{q,\mathbbm{F}}(W)=\sum_{i\in\Z}\dim_{\mathbbm{F}}(W_i)q^i.
\end{equation}

In what follows we make use of results from \cite[Section 2]{EKL} about various reductions mod 2. Consider the reduction map $\sym_n \to \sym_n \otimes_{\Z} \Z/2$.  From the definition of elementary symmetric functions it is clear that their images under this map are nonzero in $\sym_n \otimes_{\Z} \Z/2$.   In particular, $
\qrk(\sym_n)=\mathrm{dim}_{q,\Z/2}(\sym_n)$.  Similarly $\qrk(\pol_n)=\mathrm{dim}_{q,\Z/2}(\pol_n)$. The images of Schubert polynomials $\sch_w^{even}$ under the reduction map $\pol_n \to \pol_n \otimes_{\Z} \Z/2$ are nonzero and give a basis for $\pol_n \otimes_{\Z} \Z/2$ as a free $\sym_n \otimes_{\Z} \Z/2$-module.

A similar story holds in the odd setting where
\begin{equation} \nn
\qrk(\osym_n)=\mathrm{dim}_{q,\Z/2}(\osym_n) \quad \text{and} \quad \qrk(\opol_n)=\mathrm{dim}_{q,\Z/2}(\opol_n).
\end{equation}
Products of odd elementary symmetric functions provide a $\Z/2$-basis for $\osym_n \otimes_{\Z} \Z/2$, and odd Schubert polynomials provide a $\Z/2$-basis for  $\opol_n \otimes_{\Z} \Z/2$ as a free left and right $\osym_n \otimes_{\Z} \Z/2$-module.

Since the definitions of odd divided difference operators, odd elementary symmetric functions, and odd polynomials agree with their even counterparts, when reduced modulo 2, we have isomorphisms
\begin{equation} \nn
\opol_n\otimes_\Z(\Z/2)\cong\pol_n\otimes_\Z(\Z/2) \quad \text{and} \quad
\osym_n\otimes_\Z(\Z/2)\cong\sym_n\otimes_\Z(\Z/2).
\end{equation}

%
\subsection{Hecke invariants of $\opol_n$}
%

The 0-Hecke algebra acts on $\opol_n$ by odd isobaric divided difference operators $\overline{\partial}_{i}= \partial_i x_i$, see~\cite[Section 3.2]{EKL}.  Furthermore, just as the Hecke algebra $\Hq$ acts on the space of polynomials $\Z[x_1,x_2,\dots, x_n]$ we have the corresponding odd result.

\begin{prop} \label{prop_Ho-action}
The Hecke algebra $\Hq$ acts on $\opol_n$ with $T_i$ acting by the $q$-commutator
\begin{equation}
  A_i := [\partial_i,x_i]_q := \partial_i x_i -q x_i\partial_i
\end{equation}
if and only if $q=-1$ or $q=0$.
\end{prop}

\begin{proof}
It is straightforward to check that the operators satisfy the first two relations of \eqref{eq_hecke-def} for any $q$. Computing the last relation gives
\begin{equation} \nn
  A_{i+1}A_iA_{i+1} - A_{i} A_{i+1} A_{i} = (2 q + 2 q^2 ) x_1x_2 x_3 \partial_1 \partial_2 \partial_1,
\end{equation}
which is only zero when $q=-1$ or $q=0$.
\end{proof}

Note that, using the relations of the nilHecke algebra, $A_i$ can also be written as
\begin{equation} \label{eq_easyA}
  A_i = 1 - (q x_i + x_{i+1}  )\partial_i.
\end{equation}
In what follows we will abuse notation and write $T_i$ for the action of the operators $A_i$ at $q=-1$.

Observe that $f \in \osym_n$ is fixed by the action of the 0-Hecke
algebra $\Hz$ on $\opol_n$ by odd isobaric divided difference operators since
$  \overline{\partial}_i(f) = \partial_i x_i (f) = f +x_{i} \partial_i(f) =f$.

\begin{prop} \label{prop_invariants}
An odd polynomial $f \in \opol_n$ is fixed  under the action of $\Ho$ defined in Proposition~\ref{prop_Ho-action} if an only if $f$ is an odd symmetric function.
\end{prop}

\begin{proof}
For any odd symmetric function $f \in \osym_n$ we have
\begin{equation}
 T_i(f):=  \overline{\partial}_i(f) +x_i \partial_i(f)=
\overline{\partial}_i(f) =f,
\end{equation}
for all $1 \leq i \leq n-1$.  To prove the converse take an arbitrary
$g \in \opol_n$ and express it in the basis of odd
Schubert polynomials for the free left $\osym_n$-module $\opol_n$,
\begin{equation}
 g = \sum_{w \in S_n} f_w \sch_w,
\end{equation}
where $f_w \in \osym_n$.  Then
\begin{align}
 T_i(g)
 &= \sum_{w \in S_n}\partial_ix_i (f_w \sch_w)
  + x_i\sum_{w \in S_n}\partial_i( f_w \sch_w)
 \refequal{\eqref{eq_Leibniz}} g  +(x_i-x_{i+1})\sum_{w \in S_n}s_i( f_w) \partial_i(\sch_w), \nn
\end{align}
where  we used that $\partial_i$ annihilates odd symmetric functions.

We have shown that $g$ is invariant under the action of $\Ho$ if and only if
\begin{equation} \label{eq_xx1}
 \sum_{w \in S_n}s_i( f_w) \partial_i(\sch_w) =0
\end{equation}
for all $1 \leq i \leq n-1$.  Recall from \eqref{eqn_oddop_schubert} that $\partial_i(\sch_w)$ is either zero or the Schubert polynomial $\sch_{ws_i}$ if $\ell(ws_i)=\ell(w)-1$.  Therefore $g$ is invariant if and only if \begin{equation} \label{eq_xx2}
 \sum_{
\xy (0,3)*{\scs w \in S_n}; (0,-1)*{\scs \ell(ws_i)=\ell(w)-1}; \endxy }s_i( f_w) \sch_{ws_i} =0,
\end{equation}
for all $1 \leq i \leq n-1$.  As we vary over all values of $i$ the only $w\in S_n$ not appearing in \eqref{eq_xx2} for some $i$ is the identity element $e$.

Since the action of $S_n$ does not preserve the ring of odd symmetric
functions, we cannot assume that $s_i(f_w) \in \osym_n$.
However, we can still apply the standard reduction mod 2 argument from \cite{EKL} to deduce that
the $f_w =0$ for $w \neq e$. Specifically, it is not hard to check that $s_i(f_w) \in \osym_n \otimes_Z \Z/2$. Therefore, dividing \eqref{eq_xx2} by an appropriate power of 2 gives a relation
between the reductions mod two of odd Schubert polynomials contradicting that these polynomials form a basis for
$\opol_n \otimes_Z \Z/2$ as a free left $ \osym_n \otimes_Z \Z/2$-module.
Therefore, $g$ is invariant under the action of $\Ho$ if and only if $f_w=0$ for all $w\neq e$. Thus, $g= f_e \cdot \sch_{e}= f_e$ and is therefore odd symmetric.
\end{proof}

The above result identifies $\osym_n$ with the `coinvariant algebra'
for $\Ho$ under the action on $\opol_n$.  This result can also be interpreted as saying that $\Ho$ acts trivially on odd symmetric functions.  For the usual presentation of the Hecke algebra the trivial action of $T_i$ in $\Hq$ is an action by $q$.  Recall that we are using the nonstandard version of the Hecke algebra so that the trivial action corresponds to $T_i$ acting by $-q$, which for $\Ho$ amounts to $T_i$ acting by $1$.

\begin{rem}
The proof of Proposition~\ref{prop_invariants} shows that for arbitrary $q$ the odd symmetric functions are also the invariants under the operators $A_i$.
\end{rem}

%
\section{Odd cohomology of the Springer Variety}
%

%
\subsection{The odd Tanisaki ideal}
%

Let $S$ be an ordered subset of $\{x_1, x_2,
\dots, x_n\}$ and write $|S| = k$ for the size of $S$.  If
$x_i$ is in the $a$th position of  $S$ write
$S(i)=a$. Let
\begin{equation}
  x_i^S :=
   \left\{
\begin{array}{cc}
  (-1)^{S(i)-1}x_i & \text{for $i \in S$ } \\
 0 & \text{for $i \notin S$ }
\end{array}
  \right.
\end{equation}
Note that when $S=\{x_1,x_2, \dots,
x_n\}$ then $x_i^S = \tilde{x}_{i}=(-1)^{i-1}x_i$.

Define the {\em odd partial symmetric function} $\varepsilon_r^S$
corresponding to the ordered subset $S$ as follows:
\begin{equation}
 \varepsilon_r^{S} = \sum_{1 \leq i_1 < \dots < i_r \leq n} x_{i_1}^S
\dots x_{i_r}^S.
\end{equation}
Note that $\varepsilon_r^S=0$ when $r>|S|$.
Also observe that given $S\subset \{x_1, \ldots, x_{n-1}\}$ with $|S| = k$
\begin{equation}\label{SplitEqn}
\varepsilon_r^{S} = \varepsilon_r^{S\cup \{x_n\}} + (-1)^{k-1} \varepsilon_{r-1}^{S}x_n.
\end{equation}
Since $OI_{\lambda}$ is a left ideal, it convenient to write this as
\begin{equation} \label{SplitEqn2}
\varepsilon_{r}^{S} =  \varepsilon_{r}^{S\cup \{x_n\}}+(-1)^{k+r}x_n\varepsilon_{r-1}^{S}.
\end{equation}

Recall from \eqref{eq_easyA} that $\Ho$ acts on $\opol_n$ with
\[
 T_i = \partial_i x_i + x_i \partial_i = 1+(x_i - x_{i+1}) \partial_i.
\]
Below we compute the action of the operators $B_i:= (x_i - x_{i+1})\partial_i$ for $1 \leq i \leq n$
on odd partial symmetric functions.

\begin{lem} \label{lem_Baction}
The operators $B_i:= (x_i - x_{i+1})\partial_i$ for $1 \leq i \leq n$
act on odd partial symmetric functions by the formulas
\begin{equation}
 B_i(\varepsilon_r^S) =
\left\{
\begin{array}{cl}
 0 & \text{if $i,i+1 \notin S $, or $i \in S$ and $i+1 \in S$} \\
 \varepsilon_r^{S}-\varepsilon_r^{S'} & \text{if $i \in S$ and $i+1
\notin S$}\\
 \varepsilon_r^{S''}-\varepsilon_r^{S} & \text{if $i \notin S$ and
$i+1 \in S$}\\
\end{array}
\right.
\end{equation}
where $S'$ is the ordered subset obtained from $S$ by replacing $i$
with $i+1$ and $S''$ is the subset obtained from $S$ by replacing
$i+1$ with $i$.
\end{lem}

\begin{proof}
It is clear that $\partial_i(\varepsilon_r^S)$ is zero unless $i \in
S$ or $i+1 \in S$.   Suppose that both $i,i+1 \in S$.  Then $\varepsilon_r^S$ is
a sum of monomials of the form $x_{i_1}^S \dots x_{i_r}^S$.  The odd
divided difference operator annihilates any monomial in which both $i$
and $i+1$ occur  since $\partial_i(x_ix_{i+1})=0$.
Likewise, $\partial_i$ annihilates monomials containing neither $i$ nor $i+1$.
Hence, it is enough to consider monomials with exactly one element of
$\{i,i+1\}$.  Consider monomials that are identical except for $i$ and $i+1$
\begin{equation} \nn
 \partial_i\left(  x_{i_1}^S \dots x_{i_s}^S  x_i^S x_{i_{s+2}}^S
\dots x_{i_r}^S + x_{i_1}^S \dots x_{i_s}^S  x_{i+1}^S x_{i_{s+2}}^S
\dots x_{i_r}^S \right)
 = (-1)^s x_{i_1}^S \dots x_{i_s}^S  \partial_i(x_i^S +
x_{i+1}^S)x_{i_{s+2}}^S \dots x_{i_r}^S
\end{equation}
it follows that $B_i(\varepsilon_r^S)=0$ since $x_i^S + x_{i+1}^S=
(-1)^{S(i)-1}(x_i - x_{i+1})$.

Consider the case when $i \in S$ and $i+1 \notin S$. Then
\begin{equation}
 \partial_i\left( \sum_{1 \leq i_1 < \dots < i_r \leq n} x_{i_1}^S
\dots x_{i_r}^S\right)
 =
  \sum_{1 \leq i_1 < \dots < i_r \leq n} \partial_i(x_{i_1}^S \dots x_{i_r}^S)
  =
  \sum_{}(-1)^s x_{i_1}^S\dots x_{i_s}^S
\partial_i(x_{i}^S)x_{i_{s+2}}^S\dots x_{i_r}^S
\end{equation}
where the last summation is over all $1 \leq i_1 < \dots < i_r \leq n$
with some $i_{s+1} = i$ for $ 0 \leq s\leq r-1$.  Hence, the action of
$B_i$ on $e_{r}^S$ is given by
\begin{align}
(x_i-x_{i+1})\sum_{}(-1)^{s} x_{i_1}^S\dots x_{i_s}^S
(-1)^{S(i)-1}x_{i_{s+2}}^S\dots x_{i_r}^S
=
\sum_{} x_{i_1}^S\dots x_{i_s}^S (-1)^{S(i)-1} (x_i-x_{i+1})
x_{i_{s+2}}\dots x_{i_r}^S \nn .
\end{align}
The claim then follows since $x_{j}^S=x_{j}^{S'}$ for $j \neq i,
i+1$ and $i+1$ is in the same position in $S'$ as $i$ is in $S$, so
that $(-1)^{S(i)-1}x_{i+1} = (-1)^{S'(i+1)-1}x_{i+1} = x_{i+1}^{S'}$.
A similar argument proves the case when $i \notin S$ and $i+1 \in S$.
\end{proof}

\begin{defn}
Let $\lambda$ be a partition of $n$.  Consider the set of elements
\begin{equation} \label{def_OIlambda}
 O\cal{C}_{\lambda} := \left\{ \varepsilon_r^S
\;\middle\vert\;
\begin{array}{l}
    k \geq r > k -\delta_k(\lambda),  \\
    S\subseteq \{x_1, \ldots, x_n\}, |S|=k
\end{array}
 \right\}.
\end{equation}
Let $OI_{\lambda}:= \opol_n \cdot O\cal{C}_{\lambda}$ denote the homogeneous {\em (left) odd Tanisaki ideal} of $\opol_n$ generated by the elements of $O\cal{C}_{\lambda}$.
\end{defn}

We use the left ideal generated by $O\cal{C}_{\lambda}$ rather than the right ideal to be compatible with the left action of $\Ho$ on $\opol_n$.

\begin{prop}
The odd Tanisaki ideal $OI_{\lambda}$ is preserved by the action of $\Ho$, and this action preserves each homogeneous component of $OI_{\lambda}$.
\end{prop}

\begin{proof}
The second statement follows immediately from the definition of $T_i$ in terms of homogeneous elements of the odd nilHecke algebra.   Lemma~\ref{lem_Baction} and the Leibniz formula prove that $OI_{\lambda}$ is preserved under the action of $\Ho$.
\end{proof}

We define the {\em oddification
of the cohomology of the Springer variety} $OH(X^{\lambda})$ as the left $\opol_n$ graded quotient module
\begin{equation}
 OH(X^{\lambda}) := \opol_n / OI_{\lambda}.
\end{equation}
If $p,q\in \opol_n$ are equivalent modulo $OI_{\lambda}$ write $p\cong_{\lambda}q$.

\begin{cor}
The left action of $\Ho$ on $\opol_n$ restricts to a left action of $\Ho$ on
$OH(X^{\lambda})$.
\end{cor}

%
\subsection{The size of $OH(X^{\lambda})$}
%

Next we give a basis $O\mathcal{B}(\lambda)$ for $OH(X^{\lambda})$ as a free abelian group. The monomials in this basis are the same as those for $H(X^{\lambda})$ first given by De Concini-Procesi~\cite{DP} and further studied by Garsia-Procesi~\cite{GP} and others~\cite{BO,Mbirka,BFR}. Our description most closely resembles that of Mbirika~\cite[Section 2]{Mbirka}.

Given a partition $\lambda = \lambda_1 \geq \ldots \geq \lambda_n\geq0$ the height $h(\lambda)$ of $\lambda$ is the number of nonzero elements of the partition. If $\lambda$ is visualized as a top and left justified Young diagram,  the height $h(\lambda)$ is the number of rows.  For $1 \leq a\leq h(\lambda)$ and $1 \leq b \leq \lambda_a$ a node $(a,b)$ of a Young diagram is the box in row $a$ and column $b$.

For $1\leq i\leq h(\lambda)$, define $\lambda^{(i)}$ to be the partition of $n-1$ obtained from $\lambda$ by doing the following.
\begin{itemize}
\item{Remove the rightmost box in the $i^{th}$ row of $\lambda$.}
\item{If the resulting shape is top and left justified do nothing.}
\item{If removing a box creates a broken column, shift all boxes below the gap directly upwards. The resulting shape is now top and left justified.}
 \end{itemize}
Using the subpartitions $\lambda^{(i)}$ define the set $O\mathcal{B}(\lambda)$ for $OH(X^{\lambda})$ recursively by $$O\mathcal{B}(1) = \{ 1\} \textup{  \hspace{.25in} and  \hspace{.25in}} O\mathcal{B}(\lambda) = \bigsqcup_{i=1}^{h(\lambda)} \{ bx_n^{i-1} \mid b\in O\mathcal{B}(\lambda^{(i)})\}.$$

\begin{example}
Consider the partition $\lambda = (2, 1, 1, 0)$. Garsia-Procesi make use of a tree to visualize the full recursion used to build a basis. The tree corresponding to the basis $O\mathcal{B}(2,1,1)$ is constructed in Figure~\ref{fig-tree} below.

\vspace{-0.2in}
\begin{figure}[h]
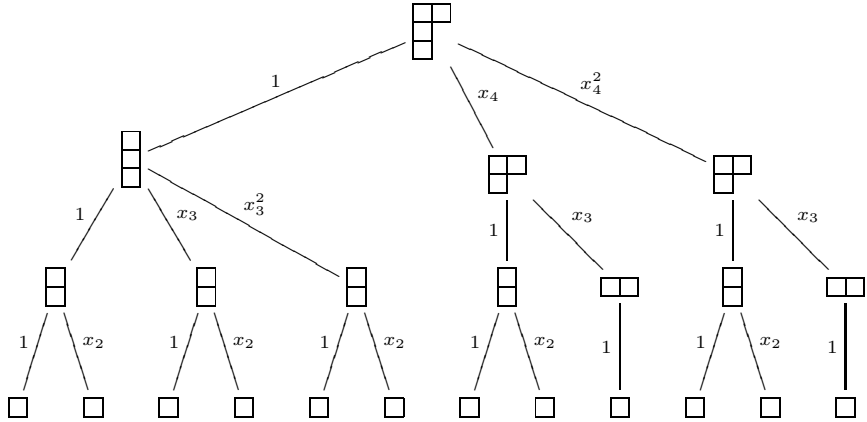

\[
 \xy
(0,44)*+{\tiny \yng(2,1,1)} = "A";
(-40,27)*+{\tiny\yng(1,1,1)} = "A1";
(10,25)* +{\tiny\yng(2,1)} = "A2";
(40,25)* +{\tiny\yng(2,1)} = "A3";
  {\ar @{-}_{1}  "A";"A1"};
  {\ar @{-}^{x_4}  "A";"A2"};
  {\ar @{-}^{x_4^2}  "A";"A3"};
(-50,10)* +{\tiny\yng(1,1)} = "A11";
(-30,10)* +{\tiny\yng(1,1)} = "A12";
(-10,10)* +{\tiny\yng(1,1)} = "A13";
   {\ar @{-}_{1}  "A1";"A11"};
   {\ar @{-}^{x_3}  "A1";"A12"};
   {\ar @{-}^{x_3^2}  "A1";"A13"};
(10,10)* +{\tiny\yng(1,1)} = "A21";
(25,10)* +{\tiny\yng(2)} = "A22";
    {\ar @{-}_{1}  "A2";"A21"};
    {\ar @{-}^{x_3}  "A2";"A22"};
(40,10)* +{\tiny\yng(1,1)} = "A31";
(55,10)* +{\tiny\yng(2)} = "A32";
    {\ar @{-}_{1}  "A3";"A31"};
    {\ar @{-}^{x_3}  "A3";"A32"};
(-55,-6)* +{\tiny\yng(1)} = "A111";
(-45,-6)* +{\tiny\yng(1)} = "A112";
(-35,-6)* +{\tiny\yng(1)} = "A121";
(-25,-6)* +{\tiny\yng(1)} = "A122";
(-15,-6)* +{\tiny\yng(1)} = "A131";
(-5,-6)* +{\tiny\yng(1)} = "A132";
(5,-6)* +{\tiny\yng(1)} = "A211";
(15,-6)* +{\tiny\yng(1)} = "A212";
(25,-6)* +{\tiny\yng(1)} = "A221";
(35,-6)* +{\tiny\yng(1)} = "A311";
(45,-6)* +{\tiny\yng(1)} = "A312";
(55,-6)* +{\tiny\yng(1)} = "A321";
    {\ar @{-}_{1}  "A11";"A111"};
    {\ar @{-}^{x_2}  "A11";"A112"};
    {\ar @{-}_{1}  "A12";"A121"};
    {\ar @{-}^{x_2}  "A12";"A122"};
    {\ar @{-}_{1}  "A13";"A131"};
    {\ar @{-}^{x_2}  "A13";"A132"};
    {\ar @{-}_{1}  "A21";"A211"};
    {\ar @{-}^{x_2}  "A21";"A212"};
    {\ar @{-}_{1}  "A22";"A221"};
    {\ar @{-}_{1}  "A31";"A311"};
    {\ar @{-}^{x_2}  "A31";"A312"};
    {\ar @{-}_{1}  "A32";"A321"};
 \endxy
\]
\caption{$O\mathcal{B}(2,1,1) = \{1, x_2, x_3, x_2x_3, x_3^2, x_2x_3^2, x_4, x_2x_4, x_3x_4, x_4^2, x_2x_4^2, x_3x_4^2\}$} \label{fig-tree}
\end{figure}
\end{example}

  \begin{lem} \label{lem_x-height}
 The monomial $x_n^{h(\lambda)}$ is an element of the odd Tanisaki ideal $OI_{\lambda}$, and thus $x_n^{h(\lambda)} \cong_{\lambda} 0$.
 \end{lem}
 \begin{proof}
Observe that the conjugate partition $\lambda' = \lambda_1' \geq \cdots \lambda_n' \geq0$ for $\lambda$ has the property that $\lambda_1' = h(\lambda)$. Thus, $\delta_{n-1}(\lambda) = \lambda_n' + \cdots + \lambda_{2}' = \delta_n(\lambda)-\lambda_1' = n-h(\lambda)$. It follows that $ h(\lambda) > (n-1) - \delta_{n-1}(\lambda)$.  Let $S = \{ x_1, \ldots, x_{n-1}\}$. Then according to the definition of the odd Tanisaki ideal, $\varepsilon_{h(\lambda)}^{S}\in OI_{\lambda}$.
Iteratively applying Equation \ref{SplitEqn2} gives
\begin{align*}
\varepsilon_{h(\lambda)}^{S}
 = \sum_{i=0}^{h(\lambda)-1} (-1)^{i\left(n-1+h(\lambda)\right) + \frac{i(i-1)}{2}}x_n^{i}\varepsilon_{h(\lambda)-i}^{S\cup \{x_n\}}
+ (-1)^{h(\lambda)(n-1)+ \frac{h(\lambda)(h(\lambda)+1)}{2}}x_n^{h(\lambda)}.
\end{align*}
Since $\varepsilon_r^{S\cup \{x_n\}}\in OI_{\lambda}$ for all $1\leq r \leq n$,  it follows that $x_n^{h(\lambda)} \cong_{\lambda} 0$.
 \end{proof}

Since $OH(X^{\lambda})$ is the quotient of $\opol_n$ by a homogenous left ideal, it has the structure of a graded left $\opol_n$-module. Let $OH(X^{\lambda})_{i}$ denote the subgroup of $OH(X^{\lambda})$ consisting of elements of the form $p x_n^{i}$ for some $p \in \opol_{n-1}$.  Lemma~\ref{lem_x-height} implies that $OH(X^{\lambda})_{h(\lambda)}$ is represented by the zero coset in $OH(X^{\lambda})$.  Hence, there is an isomorphism of abelian groups
\begin{equation}
  OH(X^{\lambda}) \cong \bigoplus_{i=1}^{h(\lambda)} OH(X^{\lambda}) _{i-1}/ OH(X^{\lambda}) _i.
\end{equation}

The following proof closely follows that of Garsia-Procesi~\cite[Proposition 2.1]{GP} (see also \cite{Tan}) but details are included for the reader's convenience.

\begin{prop}
The set $O\mathcal{B}(\lambda)$ spans $OH(X^{\lambda})$.
\end{prop}

\begin{proof}
We prove this result via induction on the size of $\lambda$. For the base case, observe that $OH(X^{1})  \cong \mathbb{Z}$. Therefore $O\mathcal{B}(1) = \{1\}$ clearly spans $OH(X^{1})$.

Assume the result holds for all partitions of $n-1$, and consider $\lambda$ a partition of $n$.
Given an element of $OH(X^{\lambda}) _{i-1}/ OH(X^{\lambda}) _i$ with representative of the form $px_n^{i-1}$ where $p\in \opol_{n-1}$, we show that
\[
px_n^{i-1} \cong_{\lambda} \left( \sum_{b\in O\mathcal{B}(\lambda^{(i)})} \alpha_b bx_n^{i-1} \right) + f x_n^i
\]
for some $\alpha_b\in \mathbb{Z}$ and $f \in\opol_n$.  By induction,  assume that $p \cong_{\lambda^{(i)}} \sum_{b\in O\mathcal{B}(\lambda^{(i)})} \alpha_b b$ and thus
\[
px_n^{i-1} \cong_{\lambda^{(i)}} \sum_{b\in O\mathcal{B}(\lambda^{(i)})} \alpha_b bx_n^{i-1}.
\]
Therefore, it only remains to show that for any $g\in OI_{\lambda^{(i)}}$  we have $$gx_n^{i-1} \cong_{\lambda} fx_n^{i}$$ for some $f\in \opol_n$. It is enough to show this for all of the odd partial elementary symmetric functions $\varepsilon_r^S$ generating the odd Tanisaki ideal $OI_{\lambda^{(i)}}$.

Recall that for $S\subset \{x_1, \ldots, x_{n-1}\}$ with $k=|S|$ we have $r>k-\delta_k\left(\lambda^{(i)}\right)$ whenever $\varepsilon_r^S \in OI_{\lambda^{(i)}}$. Because $\lambda^{(i)}$ is subordinate to $\lambda$, it follows that $\delta_k\left(\lambda^{(i)}\right) = \delta_{k+1}(\lambda)$ or  $\delta_k\left(\lambda^{(i)}\right) = \delta_{k+1}(\lambda) - 1$ for all $1\leq k \leq n-1$.

If $\delta_k\left(\lambda^{(i)}\right) = \delta_{k+1}(\lambda) - 1$, then
\[
k-\delta_k(\lambda^{(i)})
= k+1-\left(\delta_k(\lambda^{(i)})+1\right) = k+1-\delta_{k+1}(\lambda).
\]
We conclude that whenever $\varepsilon_r^S\in OI_{\lambda^{(i)}}$ it follows that $\varepsilon_r^{S\cup \{x_n\} }\in OI_{\lambda}$. By Equation \ref{SplitEqn},
$$\varepsilon_r^Sx_n^{i-1} = (\varepsilon_r^{S\cup \{x_n\}} + (-1)^{k-1} \varepsilon_{r-1}^{S}x_n)x_n^{i-1} \cong_{\lambda} \varepsilon_{r-1}^{S}x_n^i.$$

Next assume that $\delta_k\left(\lambda^{(i)}\right) = \delta_{k+1}(\lambda)$ and $r>k+1-\delta_{k+1}(\lambda)$. Then $$k+1-\delta_{k+1}(\lambda) = k+1-\delta_{k}(\lambda^{(i)}) > k-\delta_{k}(\lambda^{(i)}).$$ Since this implies $\varepsilon_r^{S} \in OI_{\lambda^{(i)}}$ and $\varepsilon_r^{S\cup \{x_n\}} \in OI_{\lambda}$, we can argue as in the previous case.

Finally assume that $\delta_k\left(\lambda^{(i)}\right) = \delta_{k+1}(\lambda)$ and $r=k+1-\delta_{k+1}(\lambda)$. Then $\varepsilon_r^S\in OI_{\lambda^{(i)}}$ but $e_r^{S\cup \{x_n\}} \notin OI_{\lambda}$. Observe that $\delta_k\left(\lambda^{(i)}\right) = \delta_{k+1}(\lambda) = \delta_{k}(\lambda) + \lambda'_{n-k-1}>\delta_{k}(\lambda) + i$. This is true since column $n-k-1$ of $\lambda$ is to the right of the column that is altered in order to obtain $\lambda^{(i)}$. Since $\lambda^{(i)}$ comes from removing the rightmost box from the $i^{th}$ row, every column to the right of that location has at most $i-1$ boxes.

Since $r=k+1-\delta_{k+1}(\lambda)$, we have that $r+i-1 = k+i-\delta_{k+1}(\lambda)>k-\delta_k(\lambda)$ and $\varepsilon_{r+i-1}^S\in OI_{\lambda}$. Notice also that $\varepsilon_{r+j}^{S\cup\{x_n\}}\in OI_{\lambda}$ for all $j\geq 1$. Rewriting Equation \ref{SplitEqn2}, we have
$x_n\varepsilon_r^S = (-1)^{k+r}\varepsilon_{r+1}^{S\cup \{x_n\}}+(-1)^{k+r+1} \varepsilon_{r+1}^S$. Iterating the equation in this form, we get
\begin{align*}
\varepsilon_r^Sx_n^{i-1} &= (-1)^{r(i-1)}\varepsilon_r^Sx_n^{i-1} \\
&= \left( \sum_{j=1}^{i-1} (-1)^{j(r+k)+\frac{(j-1)(j-2)}{2}}
 x_n^{i-j-1}\varepsilon_{r+j}^{S\cup \{x_n\}}\right)
 - (-1)^{(i-1)(r+k)+\frac{(i-2)(i-3)}{2}} \varepsilon_{r+i-1}^S\\
&\cong_{\lambda} 0.
\end{align*}

We have succeeded in showing that whenever $\varepsilon_r^S\in OI_{\lambda^{(i)}}$ it follows that
\begin{equation} \label{eq_zero-to-height}
\varepsilon_r^Sx_n^{i-1}\cong_{\lambda} fx_n^i
\end{equation}
for some $f\in \opol_n$, and so the proof is complete.
\end{proof}

\begin{thm}
 The spanning set $O\mathcal{B}(\lambda)$ is a homogeneous basis for $OH(X^{\lambda})$ as a free abelian group.
\end{thm}

\begin{proof}
Let $R$  denote the $\Z$-linear span of the set  $O\mathcal{B}(\lambda)$. Since both of the sets $\mathcal{B}(\lambda)$ and $O\mathcal{B}(\lambda)$ consist of identical monomials with unit coefficients, it follows that they give a basis for the $\Z_2$-vector spaces  $H(X^{\lambda}) \otimes_{\Z} \Z_2$ and $R \otimes_{\Z} \Z_2$, respectively.  Furthermore, this implies
\begin{equation}
\qrk H(X^{\lambda}) =  \mathrm{dim}_{q,\Z/2}\left( H(X^{\lambda}) \otimes_{\Z} \Z/2\right)
= \mathrm{dim}_{q,\Z/2}\left( R \otimes_{\Z} \Z/2 \right),
\end{equation}
so there can be no relations among the elements of  $O\mathcal{B}(\lambda)$. Any relation will, upon dividing by an appropriate power of 2, give rise to a relation in $H(X^{\lambda}) \otimes_{\Z} \Z/2$.
\end{proof}

\begin{cor}
There is an equality of ranks
\begin{equation}
  {\rm rk}_{\Z} OH(X^{\lambda}) =
  {\rm rk}_{\Z} H(X^{\lambda}) =\frac{n!}{\lambda_1! \cdots \lambda_n!}.
\end{equation}
This equality extends to an equality of graded ranks, where the right hand side above is replaced by an appropriate graded analog of the binomial coefficient, see~\cite[Equation 1.11]{GP}.
\end{cor}

\begin{rem}
While the bases $\mathcal{B}(\lambda)$ and $O\mathcal{B}(\lambda)$ consist of identical monomials,  the rewriting rules for non basis elements in $OH(X^{\lambda})$ differ from the corresponding rules in $H(X^{\lambda})$.
\end{rem}

%
\subsection{Maximal promotion} \label{sec-max}
%

We now introduce a technique for determining if a monomial multiplied by an elementary symmetric function  is in the ideal $OI_{\lambda}$.  It follows from \eqref{SplitEqn2} that for $S \subset \{ x_1, \dots, x_{n-1}\}$ and $r \geq 0$
\begin{equation} \label{eq-promo}
x_n\varepsilon_{r}^{S} = (-1)^{|S|+r}\left( \varepsilon_{r+1}^{S\cup \{x_n\}}-\varepsilon_{r+1}^{S}\right).
\end{equation}
This equation {\em promotes} the degree of elementary symmetric functions from $r$ to $r+1$.
Observe that, starting with $x_n^2\varepsilon_r^S$,  we can promote a second time by applying \eqref{eq-promo} to the term $x_n\varepsilon_{r+1}^S$.

More generally,  define the {\em maximal promotion} of
$ x_n^{a_1} x_{n+1}^{a_2} \dots x_{n+p+1}^{a_{p}}\varepsilon_{r}^{S}$
to be the result of iteratively applying \eqref{eq-promo}. Start by skew-commuting $x_n$ to the right of the other variables and promoting wherever possible until all powers of $x_n$ have been exhausted.  The unused powers of $x_n$ are then skew commuted back to the front.  Repeat this process with the variable $x_{n+1}$,  skew-commuting the unused variables back to their original location when finished maximally promoting. Continue until all variables appearing before $\varepsilon_{r}^{S}$ have been maximally promoted. As an example, the maximal promotion of $x_n^2x_{n+1}\varepsilon_r^S$ is
\begin{align*}
x_n^2x_{n+1}\varepsilon_r^S
&= x_n\left( \varepsilon_{r+2}^{S\cup \{x_n\}} - \varepsilon_{r+2}^{S\cup \{x_n, x_{n+1}\}}   \right) + (-1)^{|S|+r+3} \left(\varepsilon_{r+3}^{S\cup\{x_n, x_{n+1}\}} - \varepsilon_{r+3}^{S\cup \{x_n\}}\right) \\
&\hspace{.5in}+ (-1)^{|S|+r+3}\left(\varepsilon_{r+3}^{S\cup\{x_{n+1}\}} - \varepsilon_{r+3}^{S}\right).
\end{align*}

It is not too difficult to derive an explicit formula for maximal promotion. However, we will not need such formulas in what follows.  It suffices to know that the maximal promotion results in a sum of terms of the form $x_n^{a_1-\kappa_1} x_{n+1}^{a_2-\kappa_2} \dots x_{n+p+1}^{a_{p}-\kappa_{p}}\varepsilon_{r+|\kappa|}^{S \cup R_{\kappa}}$
where $1 \leq \kappa_i \leq a_i$, $|\kappa| = \kappa_1 + \dots + \kappa_p$, and $R_{\kappa}$ is a subset of $\{x_n, x_{n+1}, \dots, x_{n+p+1} \}$ satisfying the condition that a variable $x_j \notin R_{\kappa}$ only if $\kappa_j=a_j$.

For an analysis of whether or not $f \varepsilon_r^S$ is in the ideal $OI_{\lambda}$ for some monomial $f$, it suffices to check $\varepsilon_{r+|\kappa|}^{S\cup R_{\kappa}} \in OI_{\lambda}$ only for the minimal possible value of $r+|\kappa|$ corresponding to each variable set $S \cup R_{\kappa}$ arising in the maximal promotion.

Below we illustrate how maximal promotion can be used to identify elements in $OI_{\lambda}$.  The two lemmas appearing below are somewhat technical, but they are essential for the proof of Theorem~\ref{thm_main}.  For any $j \geq 1$ write $X_j:= \{ x_1, \dots, x_j\}$.

\begin{lem} \label{lem-tech1}
Given a partition $\lambda$ with $|\lambda|=n$ and $h(\lambda)=m$ let $d=n-\lambda_{m-1}-1$, and $T=\{x_{d+1}, x_{d+2}, \dots, x_{n}\}$, see Figure~\ref{fig-Garnir}.
Then
\begin{equation} \label{lem_promo}
x_{d+1}^{m-2} x_{d+2}^{m-2} \dots x_{n}^{m-2}\varepsilon_{\lambda_m+\alpha}^{T}\cong_{\lambda}0
\end{equation}
 for any $\alpha>0$.
\end{lem}

\begin{proof}
Write $\varepsilon_{\lambda_m+\alpha}^{T}$ as a sum of monomials and reorder the variables in \eqref{lem_promo} in increasing order resulting in a linear combination of terms of the form
\[
x_{d+1}^{a_1} x_{d+2}^{a_2} \dots x_{n}^{a_{n-d}}
\]
where $\lambda_m + \alpha$ of the exponents are $m-1$ and the rest are $m-2$. Consider the maximal promotion of
\[
x_{d+1}^{a_1} x_{d+2}^{a_2} \dots x_{n}^{a_{n-d}}\varepsilon_{0}^{X_{d}}.
\]
The result is a linear combination of terms of the form
\[
x_{d+1}^{a_1-\kappa_1} x_{d+2}^{a_2-\kappa_2} \dots x_{n}^{a_{n-d}-\kappa_{n-d}}\varepsilon_{|\kappa|}^{X_{d} \cup R_{\kappa}}
\]
with $1 \leq \kappa_j \leq a_j$ and $R_{\kappa} \subset \{x_{d+1},\dots, x_n\}$ with $x_j \notin R_{\kappa}$ only when $\kappa_j=a_j$.

For a fixed set of variables $X_{d} \cup R_{\kappa}$ a careful analysis of the smallest $r$ such that $\varepsilon_r^{X_{d} \cup R_{\kappa}}$ occurs in the maximal promotion reveals a minimal value of
\[
 r = \left( \left(h\left(\lambda\right)-2\right)\left(\lambda_{m-1}+1\right) + \lambda_m + \alpha\right)-\left(\left(h\left(\lambda\right) -3\right)|R_{\kappa}|+\omega\right)
\]
where
$\omega =
 \min\left(\lambda_m+\alpha, |R_{\kappa}|\right)$.  The first set of parenthesis is the contribution if none of the variables appear in $R_{\kappa}$ and the second set of parenthesis is the contribution if $\kappa_j=1$ for each variable $x_{j}$ appearing in $R_{\kappa}$.

Write
\[
\rho:= |X_{d} \cup R_{\kappa}| = n-(\lambda_{m-1}+1) + |R_{\kappa}|.
\]
 To show that $\varepsilon_r^{X_{d} \cup R_{\kappa}} \in OI_{\lambda}$ we must show $r> \rho-\delta_{\rho}$.  After simplification this amounts to the inequality
\[
 (h(\lambda)-1)(\lambda_{m-1}+1) +(\lambda_{m}+\alpha-\omega) > (n-\delta_{\rho}(\lambda))+(h(\lambda)-2)|R_{\kappa}|
\]
 schematically illustrated below.
 \[
\xy
(0,0)*{
   \includegraphics{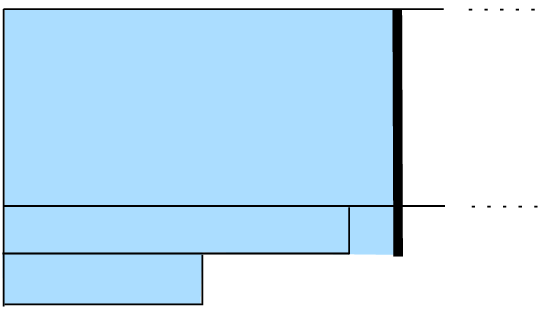}
     };
   (-20,-13)*{\scs \lambda_m};
   (-20,-8)*{\scs \lambda_{m-1}};
\endxy  +(\alpha-\omega)
\quad > \qquad
\xy
(0,0)*{
   \includegraphics{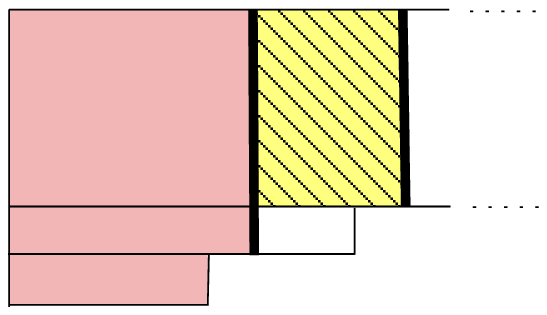}
     };
   (-20,-13)*{\scs \lambda_m};
   (-20,-8)*{\scs \lambda_{m-1}};
   (5,17)*{\overbrace{\hspace{0.6in}}};
   (5,21)*{|R_{\kappa}|};
\endxy
\]

Whenever $|R_{\kappa}| \leq \lambda_m$ it follows that $\omega=|R_{\kappa}|$. In this case, one can check that the inequality fails unless $\alpha>0$.
\end{proof}

\begin{lem} \label{lem-tech2}
Given a partition $\lambda$ with $|\lambda|=n$ and $h(\lambda)=m$ let $c= n-\lambda_{m-1}-\lambda_m$, $d=n-\lambda_{m-1}-1$, and $T=\{x_{d+1}, x_{d+2}, \dots, x_{n}\}$, see Figure~\ref{fig-Garnir}.  Further suppose that $S$ is a subset of $\{x_{c+1}, x_{c+2}\dots, x_{d} \}$ and that
$K\geq (\lambda_m+1)(h(\lambda)-2)-|S|(h(\lambda)-3)$.
Then \[
 x_{d+1}^{m-2} x_{d+2}^{m-2} \dots x_{n}^{m-2}
\varepsilon_{K+\alpha}^{X_c\cup S}\varepsilon_{\lambda_m-\alpha}^{T}\cong_{\lambda}0
 \]
 for any $\alpha>0$.
\end{lem}

\begin{proof}
The proof is similar to the previous lemma.  Write $\varepsilon_{\lambda_m-\alpha}^{T}$ as a sum of monomials, and for each summand slide these monomials to the left of $\varepsilon_{K+\alpha}^{X_c\cup S}$ acquiring a linear combination of terms of the form
\begin{equation} \label{eq_step2}
x_{d+1}^{a_1} x_{d+2}^{a_2} \dots x_{n}^{a_{n-d}} \varepsilon_{K+\alpha}^{X_c\cup S}
\end{equation}
where $\lambda_m-\alpha$ of the exponents are $h(\lambda)-1$ and the rest are $h(\lambda)-2$.
The sign resulting from sliding the monomials will not be relevant as we will show that each individual term in the sum is congruent to zero in $OI_{\lambda}$.

The maximal promotion of the expression in \eqref{eq_step2} results in a linear combination of terms of the form
\[
x_{d+1}^{a_1-\tau_1} x_{d+2}^{a_2-\tau_2} \dots x_{n}^{a_{n-d}-\tau_{n-d}} \varepsilon_{K+\alpha+|\tau|}^{X_c\cup S \cup R_{\tau}}
\]
where $1 \leq \tau_j \leq a_j$ and $x_j \notin R_{\tau}$ only when $\kappa_j=a_j$.  For a fixed set of variables $X_c\cup S \cup R_{\tau}$ one can show that the smallest value of $r$ such that $\varepsilon_r^{X_c\cup S \cup R_{\tau}}$ occurs in the maximal expansion is
\[
 r = \left( (h(\lambda)-2)(\lambda_{m-1}+1) + \lambda_m + \alpha \right) - \left( (h(\lambda)-3)|R_{\tau}| - \omega \right)
\]
where $\omega = \min\left( |R_{\tau}|, \lambda_m-\alpha\right)$. Set
\[
\rho = |X_c\cup S \cup R_{\tau}| = n-\lambda_{m-1}-\lambda_m +|S| + |R_{\tau}|.
\]Then $\varepsilon_r^{X_c\cup S \cup R_{\tau}} \in OI_{\lambda}$ if $r>\rho - \delta_{\rho}$, which after simplification amounts to the inequality
\[
(h(\lambda)-1)(\lambda_m + \lambda_{m-1}) + (\lambda_m - \omega) > (n -\delta_{\rho}) +(|S|+|R_{\tau}|)(h(\lambda)-2).
\]
When $|R_{\tau}|=\lambda_m$ it follows that $\omega = \lambda_m - \alpha$ and this inequality fails unless $\alpha>0$.
\end{proof}

%
\section{Identification of top degree components of $OH(X^{\lambda})$} \label{sec-Specht}
%

In this section we work over a field $\mathbb{F}$ for convenience.

%
\subsection{Specht modules for $\Ho$}
%

For $w\in S_n$ let $w(\mathfrak{t})$  be the action of $S_n$ on a standard tableau $\mathfrak{t}$ given by permuting the entries.  Given a standard tableau $\mathfrak{t}$ we say that $i$ precedes $j$ in $\mathfrak{t}$ if reading the entries of $\mathfrak{t}$ across rows, starting from the top left, $i$ appears before $j$. Denote by $\mathfrak{t}^{\lambda}$ the row-filled tableau of shape $\lambda$. For example, if $\lambda=(5,2,1,0,0,0,0,0)$ then
\[
 \mathfrak{t}^{\lambda} \quad = \quad
 \xy
 (0,0)*{
 \young(12345,67,8)
 }; \endxy
\]

\begin{defn}
 The Specht module $S^{\lambda}_{-1}$ of $\Ho$ is defined to be the $\mathbbm{F}$-linear span of vectors $v_{\mathfrak{t}}$ indexed by standard Young tableau $\mathfrak{t}$. The action of $\Ho$ is given by
\begin{equation} \label{eq_Specht-action}
T_i(v_{\mathfrak{t}})
=
\left\{
\begin{array}{cl}
  v_{\mathfrak{t}} & \text{if $i$ and $i+1$ are in the same row of $\mathfrak{t}$, otherwise} \\
  -v_{s_i(\mathfrak{t})} & \text{if $i$ precedes $i+1$ in $\mathfrak{t}$} \\
  -v_{s_i(\mathfrak{t})} + 2v_{\mathfrak{t}} & \text{if $i+1$ precedes $i$ in $\mathfrak{t}$}.
\end{array}
\right.
\end{equation}
(Recall our presentation of the Hecke algebra from Section ~\ref{subsec_hecke}.)
To rewrite a nonstandard tableau in terms of standards, one makes use of the following rewriting relations.
\begin{enumerate}
  \item {\bf Row relations:} If $\mathfrak{t}$ and $\mathfrak{t}'$ are identical except that  two adjacent  entries of a row have been interchanged then
\begin{equation} \nn
  v_{\mathfrak{t}} =  v_{\mathfrak{t}'}.
\end{equation}
  \item {\bf Garnir relations:} A {\em Garnir node} of $\lambda$ is a node $(a,b)$ such that $(a+1,b)$ is also a node of $\lambda$.   The {\em Garnir belt} of $\lambda$ associated to $(a,b)$ consists of those nodes in the set $A=\{ (a,f) \mid b \leq f \leq \lambda_a \}$ and $B=\{ (a+1,g) \mid 1 \leq g \leq b\}$, see Figure \ref{fig-Garnir}.

      Assume that $\mathfrak{t}$ is row standard but not standard. Then there is a column containing adjacent entries $i$ above $j$ for which $i>j$. Consider the Garnir node containing $i$.  Let $S_{A}$, $S_{B}$, and $S_{A \cup B}$ denote the subgroups of $S_n$ permuting only the entries of a tableau $\mathfrak{t}$ in $A$, $B$, and $A \cup B$, respectively.
      For $q=-1$ the Garnir relation is the sum over minimal coset representatives of $S_{A} \times S_{B} \subset S_{A \cup B}$
\begin{equation}
  \sum_{w \in S_{A} \times S_{B}/  S_{A \cup B}} (-1)^{\ell(w)}v_{w(\mathfrak{t})}=0.
\end{equation}
      As $w$ varies over all coset representatives the tableaux $w(\mathfrak{t})$ give all row standard tableaux obtained by permuting the entries $A \cup B$ in $\mathfrak{t}$.  For general $q$ the above formulas are modified by powers of $(-q)$.
\end{enumerate}
\end{defn}

Our description above follows~\cite{Mur}. See \cite{DJ1,DJ2,Mathas} for more details on Specht modules and \cite{JM} for the specific case of $q=-1$.   For a generalization to the affine Hecke algebra see \cite{Ram}.

In \cite{KLLO} an alternative presentation of the Specht module is given by generators and relations. For $\Ho$ the Specht module $S^{\lambda}_{-1}$ is freely generated by the vector $v_{t^{\lambda}}$ subject to the relations below.
\begin{enumerate}
  \item {\bf Row relations:} For all non identity permutations $w$ in the row stabilizer $W^{\lambda}$ of $\mathfrak{t}^{\lambda}$ the relation below.
\[
T_{w}(v_{\mathfrak{t}^{\lambda}}) =v_{\mathfrak{t}^{\lambda}}
\] holds.
\end{enumerate}

\begin{enumerate}\setcounter{enumi}{1}
  \item {\bf Garnir relations:} For each Garnir node $(a,b)$ of $\lambda$ the relation
  \begin{equation} \nn
  \sum_{w \in S_{A} \times S_{B}/  S_{A \cup B}} (-1)^{\ell(w)} T_{w}(v_{\mathfrak{t}^{\lambda}})=0
\end{equation}
holds.
\end{enumerate}
Hence, one need only verify the row and Garnir relations on the generating basis vector $v_{\mathfrak{t}^{\lambda}}$.
For a second root of unity one can find a slightly smaller set of Garnir relations, but we will not need this simplification here.  For a completely different set of generators and relations of the Specht module see \cite{BKW,KMR}.

%
\subsection{A bijection between monomials and standard tableaux}
%

Given a monomial in $n$ variables where each variable appears with maximum degree $n-1$, consider the composition $\mu = \mu_1, \ldots, \mu_n$ of $n$ where $\mu_i$ is the number of variables in the monomial of degree $i-1$. We will say that the monomial has \textit{shape $\mu$}. As an example, consider $m = x_1x_2^2 = x_1x_2^2x_3^0x_4^0\in \opol_4$. Then $m$ has shape $(2, 1, 1, 0).$

Given a $\lambda$-tableau $\mathfrak{t}$ let $j$ be the number occupying node $(a,b)$, and set $h(j)=a$.  Consider the map
\begin{align}
  \mathfrak{t} &\to m^{\mathfrak{t}}:=
  x_1^{h(1)-1}x_2^{h(2)-1} \dots x_n^{h(n)-1}
\end{align}
sending a $\lambda$-tableau to a monomial of shape $\lambda$.  The exponent of the variable $x_j$ is just the height minus one of the corresponding box in the tableau $\mathfrak{t}$. It is not difficult to check that this map is a bijection between the set of standard tableaux of shape $\lambda$ and elements of maximal degree in $O\mathcal{B}(\lambda)$.  This bijection is extended to a bijection between all row standard tableau and monomials of shape $\lambda$ in \cite{Mbirka}.

Denote by $m^{\lambda}$ the basis element of $OH(X^{\lambda})$ corresponding to the tableau $\mathfrak{t}^{\lambda}$. Given a part $\lambda_i$ of a partition $\lambda=\lambda_1 \geq \lambda_2 \geq \dots \geq \lambda_{h(\lambda)}$, set $p=\lambda_1 + \dots + \lambda_{i-1}$ and write
\[
x_{(\lambda_i)} := x_{p+1}x_{p+2} \cdots
x_{p+\lambda_i}.
\]
Abusing notation, we denote by $x_{(\lambda_i)}^{i-1}$ the monomial $x_{p+1}^{i-1}x_{p+2}^{i-1}\cdots x_{p+\lambda_i}^{i-1}$
so that
\begin{equation} \label{eq_mlambda}
 m^{\lambda} = x_{(\lambda_1)}^0x_{(\lambda_2)}^1 \dots x_{(\lambda_m)}^{h(\lambda)-1} =
 x_1^0 \dots x_{\lambda_1}^0 x_{\lambda_1+1}^1 \dots
  x_{\lambda_1+\lambda_2}^1 \dots x_{|\lambda|-\lambda_m}^{h(\lambda)-1} \dots x_{|\lambda|}^{h(\lambda)-1}.
\end{equation}

\begin{rem}
Our main theorem proves that incorporating certain signs in the bijection above yields an $\Ho$-module isomorphism from the Specht module $S_{-1}^{\lambda}$ to $OH(X^{\lambda})$. The sign for an arbitrary standard tableau $\mathfrak{t}$ can be deduced from the sign of $m^{\lambda}$ by acting by an appropriate $T_w$.
\end{rem}

%
\subsection{Proving Specht module relations in $OH(X^{\lambda})$}
%

%
\subsubsection{Row relations}
%

\begin{lem}[Row relations] \label{prop_row}
For all non-identity permutations $w$ in the row stabilizer $W^{\lambda}$ of $t^{\lambda}$ the relation
\[
T_{w}(m^{\lambda}) =m^{\lambda}
\]
holds in $OH(X^{\lambda})$.
\end{lem}

\begin{proof}
For each $s_i \in W^{\lambda}$ the action of $\Ho$ is given by
\begin{equation}
  T_i(m^{\lambda}) = m^{\lambda}+B_i(m^{\lambda})
  = m^{\lambda},
\end{equation}
since $B_i$ annihilates polynomials that are odd symmetric in variables $i$ and $i+1$.
\end{proof}

%
\subsubsection{Garnir relations}
%

For any Garnir node $(a,b)$ define the {\em Garnir element} associated to $(a,b)$ as
\[
G_{(a,b)}:=\sum_{w \in S_{A} \times S_{B}/  S_{A \cup B}} T_{w}(m^{\lambda})
\]in $OH(X^{\lambda})$.

\begin{figure}
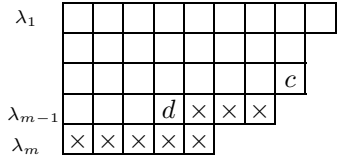

 \[
\xy
(0,0)*{
\young(\;\;\;\;\;\;\;\;\;,\;\;\;\;\;\;\;\;,\;\;\;\;\;\;\;c,\;\;\;d\times\times\times,\times\times\times\times\times)};
(-23,-9)*{\scs \lambda_m};
(-22,-5)*{\scs \lambda_{m-1}};
(-23,8)*{\scs \lambda_1};
\endxy
\]
\caption{This figure features the Young diagram of the partition $(9,8,8,7,5)$ with the Garnir belt associated with the node $(m-1,\lambda_m)$ marked with $\times$.  In the tableau $t^{\lambda}$ the number in the box labeled $c$ is $n-\lambda_{m-1}$ and the box labeled $d$ is $n-\lambda_{m-1}-1$. }\label{fig-Garnir}
\end{figure}

We first focus on the Garnir relation associated to the bottom corner of a Young diagram. For a partition $\lambda$ with $h(\lambda)=m$, this Garnir element has the form $G_{(m-1,\lambda_m)}$.
The Garnir belt associated to $G_{(m-1,\lambda_m)}$ consists of
$A=\{ (m-1,f) \mid \lambda_m \leq f \leq \lambda_{m-1} \}$ and
$B=\{ (m,g) \mid 1 \leq g \leq \lambda_m\}$.
See  Figure~\ref{fig-Garnir} for an example of a bottom corner Garnir belt.

\begin{prop} \label{prop_Garnir-rewrite}
The Garnir element corresponding to $(m-1,\lambda_m)$ can be written in the form
\begin{equation}
  G_{(m-1,\lambda_m)}= \pm
  x_{(\lambda_1)}^0 \dots
  x_{(\lambda_{m-2})}^{m-3}
   x_{(\lambda_{m-1})}^{m-2}x_{(\lambda_m)}^{m-2} \cdot \varepsilon_{\lambda_m}^{A \cup B}\cong_{\lambda}0,
\end{equation}
where $A$ and $B$ are the set of terms in the Garnir belt associated to $(m-1,\lambda_m)$.
\end{prop}

\begin{proof}

Let $Y_{(m-1,\lambda_m)}$ be the set of all row standard tableaux obtained by permuting elements of $A\cup B$ within the Garnir belt of $\mathfrak{t}^{\lambda}$ and identical to $\mathfrak{t}^{\lambda}$ off of the Garnir belt. As observed earlier, varying over all  coset representatives $w\in S_A\times S_B / S_{A\cup B}$, the collection of $w(\mathfrak{t}^{\lambda})$ is exactly $Y_{(m-1,\lambda_m)}$.

Consider $\mathfrak{t}\in Y_{(m-1,\lambda_m)}$. Let $i_1< \cdots <i_k$ be the entries in the top row of the Garnir belt in $\mathfrak{t}$ where $k = \lambda_{m-1}-\lambda_m+1$. Let $d=n-\lambda_{m-1}-1$. Then $A\cup B = \{d+1, d+2, \ldots, n\}$, see Figure \ref{fig-Garnir}.  Let  $w$ be a minimal coset representative taking $\mathfrak{t}^{\lambda}$ to $\mathfrak{t}$, i.e. $w(\mathfrak{t}^{\lambda}) = \mathfrak{t}$.  Since $\mathfrak{t}$ is row standard  $d+j\leq i_j$ for all $j$. For $q<r$ define $\sigma_{q, r} :=s_{r-1} \cdots s_{q+1}s_{q}$, and if $q=r$ define $\sigma_{q,r} = 1$. Then we may assume $w$ has the form $w := \sigma_{d+1, i_1}\sigma_{d+2,i_2}\cdots \sigma_{d+k,i_{k}}$.  A straightforward calculation shows that
\[
T_w(m^{\lambda}) =
(-1)^{m\ell(w)} x_{(\lambda_1)}^0x_{(\lambda_2)}^1 \cdots x_{(\lambda_{m-2})}^{m-3} \cdot x_{n-\lambda_{m-1}-\lambda_{m}+1}^{m-2} \cdots x_{d}^{m-2} \cdot \prod_{j=1}^{k+\lambda_m} x_{d+j}^{h(j)-1}
\]
where $h(j)$ is the height at which the entry $d+j$ appears in the tableau $w(\mathfrak{t}^{\lambda})$.

Consider $\mathfrak{t}\in Y_{(m-1,\lambda_m)}$ and its associated minimal coset representative $w\in S_A\times S_B / S_{A\cup B}$. Let $i_{k+1} < \cdots < i_n$ be the entries in the bottom row of $\mathfrak{t}$, and  let $x_w:= x_{i_{k+1}} \cdots x_{i_n}$. Then we can write
\[
 \prod_{j=1}^{k+\lambda_m} x_{d+j}^{h(j)-1}
  = (-1)^{m\ell(w)}\left(\prod_{j=1}^{k+\lambda_m} x_{d+j}^{m-2}\right) x_w.
\]
Putting these together we get the following expression for the Garnir element
\begin{align*}
G_{(m-1,\lambda_m)} &= \sum_{w \in S_{A} \times S_{B}/  S_{A \cup B}} (-1)^{\ell(w)}T_{w}(m^{\lambda}) \\
&=  x_{(\lambda_1)}^0x_{(\lambda_2)}^1 \cdots x_{(\lambda_{m-2})}^{m-3}x_{(\lambda_{m-1})}^{m-2}x_{(\lambda_m)}^{m-2}\sum_{w \in S_{A} \times S_{B}/  S_{A \cup B}}(-1)^{\ell(w)} x_w.
\end{align*}

It only remains to verify that the signed sum of terms $x_w$ is the appropriate odd partial elementary symmetric function. Given two coset representatives $w=s_i\cdot w'$ their monomials $x_w$ and $x_{w'}$ are identical except that one contains $x_i$ while the other contains $x_{i+1}$. Furthermore $\ell(w) = \ell(w')+1$, so $x_w$ and $x_{w'}$  are given opposite sign in the Garnir element $G_{(m-1,\lambda_m)}$. This is exactly the sign convention in $\varepsilon_{\lambda_m}^{A\cup B}$, so $\sum_{w \in S_{A} \times S_{B}/  S_{A \cup B}}(-1)^{\ell(w)} x_w = \pm \varepsilon_{\lambda_m}^{A\cup B}$.
\end{proof}

\begin{prop} \label{prop_mlambdam}
Given a partition $\lambda$ with $|\lambda|=n$ and $h(\lambda)=m$, then
$G_{(m-1,\lambda_m)} \cong_{\lambda} 0$.
\end{prop}

\begin{proof}
Let $c= n-\lambda_{m-1}-\lambda_m$ and $d=n-\lambda_{m-1}-1$, as in Figure~\ref{fig-Garnir}.  Recall that $X_c:=\{x_1,x_2,\dots, x_c \}$.  Consider the monomial $x_{c+1}^{m-2} x_{c+2}^{m-2} \dots x_{d}^{m-2}$ corresponding to those boxes in row $(m-1)$ not involved in the Garnir relation.  Maximally promoting $x_{c+1}^{m-2} x_{c+2}^{m-2} \dots x_{d}^{m-2}\varepsilon_0^{X_{c}}$ results in a linear combination of terms of the form
\[
x_{c+1}^{m-2-\kappa_1} \dots x_{d}^{m-2-\kappa_{\lambda_m-1}}\varepsilon_{|\kappa|}^{X_c\cup R_{\kappa}}
 \]
with $1 \leq \kappa_j \leq m-2$,
as was explained in Section~\ref{sec-max}.  Then by Proposition~\ref{prop_Garnir-rewrite} (after rearranging terms) we have
\[
G_{(m-1,\lambda_m)}= \sum \alpha_{\kappa} x_{(\lambda_1)}^0 \dots x_{(\lambda_{m-2})}^{m-3}
  x_{(\lambda_{m-1})}^{m-2} x_{c+1}^{m-2-\kappa_1} \dots x_{d}^{m-2-\kappa_{\lambda_m-1}} x_{d+1}^{m-2} \dots x_n^{m-2}
  \varepsilon_{|\kappa|}^{X_c\cup R_{\kappa}}\varepsilon_{\lambda_m}^{A \cup B}
\]
where the sum is over all terms in the maximal promotion. The precise form of the value of the coefficients $\alpha_{\kappa}$ will not be relevant here.

The result follows by proving that
$x_{d+1}^{m-2} \dots x_n^{m-2}  \varepsilon_{|\kappa|}^{X_c\cup R_{\kappa}}\varepsilon_{\lambda_m}^{A \cup B} \cong_{\lambda}0$ for each term appearing in the maximal promotion. To see this note that for any partition of a subset  $S \subset \{ x_1,\dots, x_n\}$ into ordered subsets $U$ and $V$ it is evident that
\begin{equation}
  \varepsilon_r^S = \sum_{j=0}^r (-1)^{|U| \cdot j}e_{r-j}^Ue_{j}^V.
\end{equation}
Then we have
\[
\varepsilon_{|\kappa|}^{X_c\cup R_{\kappa}}\varepsilon_{\lambda_m}^{A \cup B}
= (-1)^{|X_c \cup R_{\kappa}|\cdot\lambda_m}
\left(\varepsilon_{|\kappa|+\lambda_m}^{X_{c} \cup R_{\kappa} \cup A \cup B}
- \sum
(-1)^{|X_c \cup R_{\kappa}|\cdot j}
\varepsilon_{|\kappa|+\lambda_m-j}^{X_c\cup R_{\kappa}}\varepsilon_{j}^{A \cup B}
  \right)
\]
where the sum is over all $0 \leq j \leq |\kappa|+\lambda_m$ with $j \neq \lambda_m$.  The reader can easily verify that
\[
\varepsilon_{|\kappa|+\lambda_m}^{X_{c} \cup R_{\kappa} \cup A \cup B} \in OI_{\lambda}.
 \]
Using Lemma \ref{lem-tech1} when $j > \lambda_m$ and Lemma~\ref{lem-tech2} when $j<\lambda_m$, it follows that each term in the sum is also congruent to zero in $OI_{\lambda}$.
\end{proof}

\begin{prop} \label{prop_Garnir}
Given a Garnir node $(a,b)$ of a partition $\lambda$ then
\begin{equation} \label{eq-Garnier}
G_{(a,b)} \cong_{\lambda} 0
\end{equation}
in  $OH(X^{\lambda})$.
\end{prop}

\begin{proof}
If $(a,b)$ is not a bottom corner Garnir node, let $\lambda'$ be the subpartition of $n'<n$ obtained by removing all boxes below and to the right of the Garnir belt in $\lambda$ so that the node $(a,b) = (h(\lambda')-1, \lambda'_{h(\lambda')})$ is a bottom corner in $\lambda'$.  It follows from Proposition~\ref{prop_mlambdam} that $G_{(a,b)} \cong_{\lambda'} 0$. Let $\mu$ be the partition subordinate to $\lambda$ obtained by adding a box to $\lambda'$ of minimal height. Then arguing as in \eqref{eq_zero-to-height} if follows that $G_{(a,b)}x_{n'+1}^{h(\mu)-1} \cong_{\mu} f \cdot x_{n'+1}^{h(\mu)}$ for some $f \in OH(X^{\mu})$. Continuing in this way the Garnir relation for the node $(a,b)$ in $\lambda$ follows from the Garnir relation for this node in $\lambda'$.
\end{proof}

%
\subsubsection{Main theorem}
%

Propositions~\ref{prop_row} and \ref{prop_Garnir} prove our main result.

\begin{thm} \label{thm_main}
There is an isomorphism $S_{-1}^{\lambda} \cong OH^{\topp}(X^{\lambda}) \otimes_{\Z} \mathbbm{F}$ of $\Ho$-modules sending the generator $v_{\mathfrak{t}^{\lambda}}$ corresponding to the row filled tableau in the Specht module $S_{-1}^{\lambda}$ to the monomial $m^{\lambda}$.
\end{thm}

%
%



%

%
%
%
%
%

%
\end{document}